\documentclass[11pt, reqno]{amsart}
\usepackage{amsmath,amssymb,amsfonts,amscd,hyperref,color}
\usepackage[utf8]{inputenc}
\usepackage{csquotes}

\usepackage[abs]{overpic}
\usepackage{verbatim}

\newcommand{\Hmm}[1]{\leavevmode{\marginpar{\tiny%
$\hbox to 0mm{\hspace*{-0.5mm}$\leftarrow$\hss}%
\vcenter{\vrule depth 0.1mm height 0.1mm width \the\marginparwidth}%
\hbox to
0mm{\hss$\rightarrow$\hspace*{-0.5mm}}$\\\relax\raggedright #1}}}

\newtheorem{thm}{Theorem}[section]

\newtheorem{lem}[thm]{Lemma}
\newtheorem{pro}[thm]{Proposition}

\theoremstyle{definition}

\newtheorem*{rem}{Remark}

\numberwithin{equation}{section}
\newcommand{\Z}{{\mathbb Z}}
\newcommand{\R}{{\mathbb R}}

\newcommand{\N}{{\mathbb N}}

\newcommand{\D}{{\mathbb D}}

\let\L\undefined

\newcommand{\L}{L}

\newcommand{\al}{{\alpha}}
\newcommand{\be}{{\beta}}

\newcommand{\eps}{{\varepsilon}}

\newcommand{\ka}{{\kappa}}

\renewcommand{\d}{\,\mathrm{d}}
\newcommand{\ep}{\varepsilon}

\newcommand{\longto}{\longrightarrow}

\newcommand{\s}{\sigma}
\renewcommand{\L}{\mathcal{L}}
\renewcommand{\D}{\Delta}

\begin{document}
\title[Optimal Hardy Inequality for Fractional Laplacians]
{Optimal Hardy Inequality for Fractional Laplacians on the Lattice}
\author[P.~Hake]{Philipp Hake}
\author[M.~Keller]{Matthias Keller}
\author[F.~Pogorzelski]{Felix Pogorzelski}
\address{P.~Hake and F.~Pogorzelski, Institut f\"ur Mathematik, Universit\"at Leipzig
\\04109 Leipzig, Germany}

\address{M.~Keller, Institut f\"ur Mathematik, Universit\"at Potsdam
14476  Potsdam, Germany}

 \email{philipp.hake@math.uni-leipzig.de}
\email{felix.pogorzelski@math.uni-leipzig.de}
\email{matthias.keller@uni-potsdam.de}

\date{\today}

\begin{abstract}
We study the fractional Hardy inequality on the integer lattice. We prove null-criticality of the Hardy weight  and hence optimality of the  constant. More specifically, we present a family of Hardy weights with respect to a parameter and show that below a certain threshold the Hardy weight is positive critical while above the threshold it is subcritical. In particular, the Hardy weight at the threshold is optimal in the sense that any larger weight would fail to be a Hardy weight and the Hardy inequality does not allow for a minimizer. A crucial ingredient in our proof is an asymptotic expansion of the fractional discrete Riesz kernel.
\end{abstract}
\maketitle


\section{Introduction}

The fractional Laplacian,  originally introduced by Riesz \cite{R39}, arises naturally in several contexts, including relativistic quantum mechanics, turbulence, elasticity, laser physics, and anomalous transport \cite{MV,B,DPV,L,M}. In recent years, fractional Laplacians on discrete spaces have attracted considerable attention, see e.g.\@ \cite{CR18,CRSTV,das_fuente-fernandez_2025,DMRM,D25,FRR,GRM,gerhat2025criticality}.
In this note, we study the Hardy inequality for the fractional Laplacian on the integer lattice which complements the intensive study of  Hardy inequalities in recent years for the classical Laplacian on lattice and half-line graphs or trees, see e.g.\@ \cite{BSV21,GKS25, Gup23, Gup24, FKP23, FR25, KL16, KL23, KLS22, KPPHardy,KPP18,KS22,SW25} and references therein.

\noindent
For the continuous case, Hardy inequalities involving the fractional Laplacian are well studied, with optimal constants known \cite{BDK16,FLS08,H77,Y99} and analogous results available for the half-space \cite{BD08}. In the discrete setting, however, the fractional Hardy inequality obtained in \cite{CR18} left the optimality of its constant as an open problem.  
In contrast to the case of the classical  Laplacian, where optimal Hardy weights are studied in \cite{Gup23, KL23,KPP20} for $\Z^d$, optimality of the constant and the weight of the Hardy inequality in the fractional case  is so far only known for $d=1$ \cite{KN,das_fuente-fernandez_2025}. 

A major challenge in dealing with the discrete fractional Laplacian on $\Z^d$ is  the lack of spherical symmetry of the operator, which makes an analysis such as pursued in \cite{FLS08} for the continuum infeasible. On the other hand, it seems hard to deduce any explicit asymptotics of the Hardy weight nor of the constant for the supersolution construction from \cite{KPP20}, even when extended to non-locally finite graphs. Furthermore, the one-dimensional proof of \cite{KN} depends on an explicit construction of a null-sequence which does not extend  to arbitrary dimensions. 

Thus, we come up with a new strategy to prove optimality of the Hardy weight and constant. This approach is fundamentally based on criticality theory and uses a family of Hardy weights depending on a parameter $\alpha$ as already introduced in \cite{CR18}. The main difficulty lies in determining the criticality type of the Hardy weight at the threshold value of $\alpha$ where the constant of the Hardy weight matches that of an optimal Hardy weight. Here, a  characterization of positive critical Hardy weights via potential theory from the doctoral thesis of the first author \cite{hake2025optimal} plays a crucial role. Building on this characterization, we use the ground states of positive critical Hardy weights in the family to construct a null-sequence for the null-critical and hence optimal Hardy weight. Throughout the analysis, we take advantage of  precise asymptotics of the fractional discrete Riesz kernel on $\Z^d$ (cf. Theorem~\ref{thm:RieszAsymptotics}) which might be of independent interest.

 Let us now describe our setting and main result in more detail. We consider the standard discrete Laplacian $\Delta$ on $\Z^d$, $d\in \N$, given by
\[
\Delta f(x) = \sum_{|y-  x|=1} (f(x) -f(y)),
\]
which is a bounded operator on $\ell^2(\Z^d)$. For $\sigma \in (0,1)$, the fractional Laplacian $\Delta^\sigma$ is defined via the spectral theorem and can be represented as
\[ \Delta^{\sigma}f(x)=\frac{1}{|\Gamma(-\sigma)|}\int_{0}^{\infty}(I-e^{-t\Delta})f(x)\frac{\d t}{t^{1+\sigma}} . \]

We say a function $ w:\mathbb{Z}^d\to [0,\infty) $ is a \emph{Hardy weight} for the fractional Laplacian if the Hardy inequality
\begin{align*}
	\langle \Delta^{\sigma}\varphi,\varphi\rangle \geq \sum_{x\in \Z^d}w(x)|\varphi(x)|^2	
\end{align*}
holds for all $ \varphi\in C_c(\mathbb{Z}^d) $. If the inequality holds,  the inequality can be extended to all $ \varphi $ in 
$$\mathcal{D}_0^\sigma=\overline{C_c(\Z^d)}^{\|\cdot\|_{\sigma,0}},$$ where closure is taken with respect to $\|\varphi\|_{\sigma,0}=\sqrt{\langle \Delta^\sigma \varphi,\varphi\rangle+|\varphi(0)|^2}$.

A Hardy weight $ w $ is called \emph{optimal} if for every Hardy weight $w'\ge w $ we have $w'=w$ and the Hardy inequality does not allow for a minimizer in $\mathcal{D}_0^\sigma.$ This definition goes back to Hardy weights in the continuum \cite{DFP14} and as will be discussed below includes also sharpness of a constant.

To formulate the main result we need a bit more notation. 
For  $ \al\in (-d/2,1] $,   we define the Riesz kernel $\kappa_{\al}:\Z^d\to [0,\infty)$ by $ \kappa_{\al}(0)=0 $ for $ \al>0 $, $ \kappa_{0}=1_{0} $, $\kappa_1(x)=1_{|x|=1}$ and otherwise by
$$  \ka_{\al}(x)=\frac{1}{|\Gamma(-\al)|}\int_{0}^{\infty}e^{-t\Delta}1_{0}(x)\frac{\d t}{t^{1+\al}}.  $$
Here, $ 1_{y} $ denotes the characteristic function of $ y\in \Z^d $. 
Since the semigroup $e^{-t\Delta}$ is  positivity improving (cf.\@ \cite[Theorem~1.26]{KLW}), we infer that  $ \kappa_{\al} $ is strictly positive apart from $x=0$ and  $\alpha\notin\{0,1\}$. One can show that $\kappa_\alpha$ is indeed finite and has the following asymptotics (cf. Theorem~\ref{thm:RieszAsymptotics}) as $x\to\infty$,
\[
\kappa_\alpha(x) = C_{d,\alpha} \, |x|^{-d-2\alpha} + O\left(|x|^{-d-2\alpha-2}\right),
\]
where 
\[
C_{d,\alpha} =
\frac{4^{\alpha} \Gamma(d/2+\alpha)}{\pi^{d/2}|\Gamma(-\alpha)|}
\]
for $\alpha\notin \{ 0,1 \}$.

Furthermore, we define the weights for $ \sigma\in (0,1] $ and $ \al\in (\sigma, d/2) $  by
\begin{align*}
	w_{\sigma,\al}=\frac{\kappa_{\sigma-\al}}{\kappa_{-\al}}.
\end{align*}
We observe that $w_{\sigma,\alpha} $ is strictly positive as $\kappa_{\sigma-\alpha}$ and $\kappa_{-\alpha}$ are strictly positive in the respective ranges  of $\alpha$ and $\sigma$.

\medskip

The main result of this paper is the following theorem.

\begin{thm}\label{t:main_result} 
	Let $ d\in \N $,  and $ \sigma\in (0,1] $ with $\sigma<d/2$ if $d \in \{1,2\}$. Then,  for $\alpha_0= \frac{d/2 +\sigma}{2} $, the function	
	$w_{\sigma,\al_0} $ is an optimal Hardy weight satisfying 
	\begin{align*}
		w_{\sigma,\al_0}(x)=\frac{c_{d,\sigma}}{|x|^{2\sigma}}+O\left(\frac{1}{|x|^{2\sigma+2}}\right)
	\end{align*}
	 with the optimal constant 
	\begin{align*}
		c_{d,\sigma}=4^{\sigma}\frac{\Gamma(\frac{d}{4}+\frac{\sigma}{2})^2}{\Gamma(\frac{d}{4}-\frac{\sigma}{2})^2}.
	\end{align*}
\end{thm}

\begin{rem} Let us highlight that the result also includes an optimal Hardy weight for the classical case of $\Delta = \Delta^1$. Naturally, for $\sigma=1$, $d \geq 3$, and $\alpha_0 = (d/2 + 1)/2$,  the first order term of $w_{\sigma,\alpha_0}$ coincides with the first order term of the optimal Hardy weight for $\Delta$ on $\Z^d$ found in \cite[Theorem~7.2]{KPP18}, with optimal constant $c_{d,1} = (d-2)^2/4$. It would be interesting to compare the higher order terms of these two Hardy weights, since in the one dimensional case of $\N_0$ the construction via fractional Laplacians applied in \cite{das_fuente-fernandez_2025} yields better higher order terms than the classical supersolution construction \cite{KPPHardy}.
\end{rem}

The theorem is proven in Section~\ref{s:proof}. Indeed, we prove a much more structural result in Theorem~\ref{t:summary} below which summarizes our findings on the family of Hardy weights $w_{\sigma,\alpha}$. In particular, we show that there are three regimes for the parameter $\alpha$ in relation to $\alpha_0=(d/2+\sigma)/2$:
 \begin{itemize}
	\item[(a)] for $\sigma <\alpha<\al_0 $, the weight $ w_{\sigma,\al} $ is positive critical,
		\item[(b)] for $ \al=\alpha_0 $, the weight $ w_{\sigma,\al} $ is null-critical,
		\item[(c)] for $ \alpha_0 <\al<d/2$, the weight $ w_{\sigma,\al} $ is subcritical.
 \end{itemize}
 To this end recall that $w$ is called \emph{subcritical} if there exists a Hardy weight $w'\ge w$ with $w'\neq w$ and {\em critical} otherwise. A critical Hardy weight $w$ is called is \emph{positive critical} if the operator $\Delta^\sigma - w$ admits a ground state $v$ (i.e., a positive harmonic function which is unique up to scaling) in $\mathcal{D}_0^\sigma$ and \emph{null-critical} if the ground state $v$ is not in $\mathcal{D}_0^\sigma$. This is indeed equivalent to the ground state $v$ being in  $\ell^2(\Z^d,w)$ or not, cf. \cite{KPP20}. So, in the positive critical case, the Hardy weight still cannot be improved, but there is a minimizer $v$ which is  not a minimizer  in the null-critical case. Furthermore, in the case of a null-critical Hardy weight $w$, there is also no $\lambda>0$ such that $(1+\lambda)w$ is a Hardy weight on $C_c(X\setminus K)$ for any finite set $K\subset X$, cf. \cite{KPP20,KN,Fischer,hake2025optimal,KovarikPinchover}. This is called \emph{optimal at infinity} in \cite{DFP14,DP16,KPP20}. Hence, the constant in front of the leading asymptotic cannot be improved which justifies the notion of optimality in Theorem~\ref{t:main_result}.

The paper is structured as follows. In the next section we present a representation of the fractional Laplacian and the Green's function in terms of the fractional discrete Riesz kernel. In Section~\ref{s:proof} we prove our main result, Theorem~\ref{t:main_result} which is a consequence of the more structural Theorem~\ref{t:summary} discussed above. Finally, in Section~\ref{s:Landis} we give an application in form of a Landis type theorem that can be understood as a unique continuation result at infinity. In the appendix we present the asymptotic expansion of the fractional discrete Riesz kernel, slightly extending the known results  which might therefore be of independent interest.

\section{Representation of the Fractional  Laplacian and the Green's function}

In this section we take a closer look at the fractional Laplacian and the Green's function. In particular we observe that the fractional Laplacian is a nearest neighbor operator of a weighted non-locally finite graph, cf. \cite{CR18,CRSTV,das_fuente-fernandez_2025,DMRM,FRR,GRM,KN}. Although many of the results in this section are known to experts, we include them for the convenience of the reader.

Recall the definition of the Riesz kernel $\kappa_\alpha$ from above and that due to translation invariance of $\Delta$, we have for $\alpha\notin \{0,1\}$ and $x,y \in \Z^d$
\begin{align*}
	\kappa_{\alpha}(x-y) = \frac{1}{|\Gamma(-\alpha)|} \int_0^{\infty} e^{-t\Delta}1_y(x) \frac{\d t}{t^{1+\al}}.
\end{align*}

\begin{thm}[Fractional Laplacian] \label{thm:graphlaplacian}
	For $ \sigma\in (0,1]$, the fractional Laplacian $ \Delta^{\sigma} $ can be represented as
	\begin{align*}
		\Delta^{\sigma}f(x)= \sum_{y\in \Z^d}\kappa_{\sigma}(x-y)(f(x)-f(y)),
	\end{align*}
	for $ f\in \ell^2(\Z^d)$ and $ x\in \Z^d $.
\end{thm}

 \begin{proof}
Observe that since $\Delta$ is the generator of a Markovian semigroup it extends to $\ell^\infty$ and in particular by its positivity preserving property, monotone convergence and stochastic completeness, we obtain $1=e^{-t\Delta}1=\sum_{y\in \Z^d}e^{-t\Delta}1_y $.
	We calculate for $ \sigma\in (0,1)$
	\begin{multline*}
	\Delta^\sigma f(x) 
		= \frac1{|\Gamma(-\sigma)|} \int_0^\infty \left[
			\left((e^{-t\Delta}1)(x)f(x) - (e^{-t\Delta}f)(x)\right) 
			\right]\frac{\d t}{t^{1+\sigma}} \\ 
		= \frac1{|\Gamma(-\sigma)|} \int_0^\infty 
			\left[\sum_{y\neq x}e^{-t\Delta}1_y(x)(f(x)-f(y)) \right]\frac{\d t}{t^{1+\sigma}}.  
	\end{multline*}		
	Now by	using Fubini's theorem 			which is justified by the asymptotic behavior of the Riesz kernel (cf. Theorem~\ref{thm:RieszAsymptotics}), we conclude
		\begin{align*} 
	\ldots	&=  \sum_{y\neq x}\left(\frac1{|\Gamma(-\sigma)|} \int_0^\infty  
			e^{-t\Delta}1_y(x)\frac{\d t}{t^{1+\sigma}}\right)(f(x)-f(y)) \\ 
		&=  \sum_{y\in \Z^d} \kappa_\s(x-y) (f(x)-f(y)).
\end{align*}
The case $  \sigma=1$ is trivial by definition of $\kappa_1(x)=1_{|x|=1}$, $x\in \Z^d$.
 \end{proof}
 The theorem above and the asymptotics of the Riesz kernel, Theorem~\ref{thm:RieszAsymptotics}, show that we can extend the fractional Laplacian $\Delta^{\sigma}$ to $\ell^1(\Z^d,(1+|\cdot|)^{-d-2\sigma}) $. In particular, again by Theorem~\ref{thm:RieszAsymptotics}, we  have $\kappa_{-\alpha}\in \ell^1(\Z^d,(1+|\cdot|)^{-d-2\sigma})$ for $-d/2<\alpha<1$, so we can apply the Laplacian to $\kappa_{-\alpha}$. 
 
The \emph{Green's function} $G^\sigma:\Z^d\times \Z^d\to(0,\infty]$ of the fractional Laplacian $ \Delta^{\sigma} $, $ \sigma\in (0,1] $, is given by
\begin{align*}
	G^\sigma(x,y)=\lim_{E\searrow 0} (\Delta^{\sigma}+E)^{-1}1_y(x)=\int_{0}^{\infty}e^{-t\Delta^\sigma}1_{y}(x){\d t}.
\end{align*}

If $G^\sigma$ is finite, i.e.,\@ the associated graph is transient, then we denote the Green operator $G^\sigma $ on $ \mathcal{G}^{\sigma}= \{ f:\Z^d\to\R\mid \sum_{y\in \Z^d}G^{\sigma}(x,y)\vert f(y)\vert < \infty \} $ by
\begin{align*}
G^{\sigma}f(x)=\sum_{y\in \Z^d}G^{\sigma}(x,y)f(y).
\end{align*}

A function $u$ is called a \emph{potential} if $\Delta^\sigma u \in \mathcal{G}^{\sigma}$ and $u = G^\sigma \Delta^\sigma u$.

The first part of the next proposition was already shown in \cite{KN} for $d=1$ and the proof  carries over verbatim to our setting. For the convenience of the reader we include a proof.

\begin{pro}\label{p:exponents}
	Let $ \sigma\in (0,1] $ and $\alpha \in [\sigma, d/2)$. Then
	\begin{align*}
		\Delta^{\sigma}\kappa_{-\alpha}=\kappa_{-\alpha+\sigma} .
	\end{align*} 
	Furthermore, $ \kappa_{-\alpha} $ is a superharmonic potential with
	\begin{align*}		\kappa_{-\alpha} = G^\sigma\kappa_{-\alpha+\sigma} 	,
	\end{align*}
	and, in particular, the Green's function is given by $G^\sigma(x,y) = \kappa_{-\sigma}(x-y)$ for $x,y\in \Z^d$.
	Moreover,
	\begin{align*}
		w_{\sigma,\al}=\frac{\Delta^\sigma \kappa_{-\alpha} }{\kappa_{-\alpha} }=\frac{\kappa_{\sigma-\al}}{\kappa_{-\al}}=\frac{\kappa_{\sigma-\al} }{G^\sigma\kappa_{\sigma-\alpha}}
	\end{align*}
	is a Hardy weight.
\end{pro}
 \begin{proof}
	We define $\kappa_{\beta,\ep} : \Z^d \longto [0,\infty]$ for $\beta\in (-\infty,1]$ and $\ep>0$ as 
		\[\kappa_{\beta,\eps} 
		= \frac1{|\Gamma(-\beta)|} \int_0^\infty 
		e^{-t\ep} e^{-t\D}1_{0} \frac{\d t}{t^{1+\beta}} \] 
		for $ \beta\notin\{0 ,1\}$, $ \ka_{0,\ep}=1_{0} $ for $ \beta=0 $ and $ \ka_{1,\ep}(x)=1_{|x|=1} $ for $\beta=1$ and $x\in\Z^d$. 	
		Spectral calculus gives for $ \beta \leq 0 $ 
		\begin{align*}
			 (\D+\ep)^{\beta} 1_{0}=\kappa_{\beta,\eps} 	
		\end{align*}
	and, therefore, $ \ka_{\beta,\eps}\in \ell^{2}(\Z^d) $  for $ \beta \leq 0 $. 
	
	Moreover, for $ 0 < \beta < 1 $ and $ f\in \ell^{2}(\Z^d) $, we obtain  by the spectral calculus
	\[ (\Delta+\eps)^{\beta}f(x)=\frac{1}{|\Gamma(-\be)|}\int_{0}^{\infty}(1-e^{-t\eps}e^{-t\Delta})f(x)\frac{\d t}{t^{1+\beta}}. \] 
		We use $ \int_{0}^{\infty}(1-e^{-t\eps})t^{-1-\be}\d t=\eps^{\beta}|\Gamma(-\beta)| $ and $ e^{-t\Delta}1=1 $ to conclude with the argument as in the proof of Theorem~\ref{thm:graphlaplacian} 
		\begin{align*}
			(\Delta+\eps)^{\beta}f(x) &=\frac{1}{|\Gamma(-\be)|}\int_{0}^{\infty}e^{-t\eps}(1-e^{-t\Delta})f(x)\frac{\d t}{t^{1+\beta}} +\eps^{\be}f(x) \\
		&=\sum_{y\in \Z^d}\kappa_{\beta,\eps}(x-y)(f(x)-f(y))+\eps^{\beta} f(x),
	\end{align*}
	and the corresponding equality of the left and right hand side holds trivially for $\beta=1$.
Now, let $0<\s\leq\alpha<d/2$. By the  spectral calculus and the above cases for $\beta$, we see that 
		\begin{align*}
			 (\D+\ep)^\s \kappa_{-\alpha,\eps} 
		= (\D+\ep)^\s (\D+\ep)^{-\alpha} 1_{0}
		= (\D+\ep)^{-(\alpha-\s)} 1_{0}
		=\kappa_{\s-\al,\eps}.
		\end{align*}
		Further,	$\ka_{\beta,\eps}(x)$ 	 converges monotonously to $ \kappa_{\beta} (x)$, $ x\in \Z^d $, as $ \eps\to0 $. This and the monotone convergence theorem give the validity of the first statement. \\
		In particular, this shows that $\kappa_{-\alpha}$ is superharmonic and it follows from Theorem~\ref{thm:RieszAsymptotics} that $\kappa_{-\alpha} \in C_0(\Z^d)$. The existence of a non-trivial superharmonic function in $C_0(\Z^d)$ implies the transience of the graph, cf.\@ \cite[Theorem~6.1]{KLW}, 
		and on a transient graph every superharmonic function in $C_0(\Z^d)$ is a potential (\cite[Corollary~1.39]{hake2025optimal}). Ergo, $\kappa_{-\alpha}$ is a potential, which means, by definition, that $\Delta^\sigma \kappa_{-\alpha} = \kappa_{-\alpha + \sigma} \in \mathcal{G}^\sigma$ and $\kappa_{-\alpha} = G^\sigma \Delta^\sigma \kappa_{-\alpha} = G^\sigma \kappa_{-\alpha + \sigma}$. For $\alpha = \sigma$, this yields $\kappa_{-\sigma} = G^\sigma \kappa_0 = G^\sigma 1_0 = G^\sigma(\cdot , 0)$. Since $G^\sigma$ is translation invariant, it holds that $\kappa_{-\sigma}(x-y) = G(x,y)$.

The final statement follows now directly by the ground state transform, see e.g. \cite[Proposition~4.8]{KPP20}, \cite[Lemma~4.8]{KLW} or \cite[Theorem 10.1]{FLW} since $\kappa_{-\alpha}$ is harmonic for the operator $\Delta^\sigma - w_{\sigma,\alpha}	$.
\end{proof}

The corollary above gives us a family of Hardy weights. We now take a closer look at their asymptotic behavior and monotonicity.
\begin{pro}\label{t:asymptotics_weight}
For $ \sigma\in (0,1] $ and $ \al\in (\sigma,d/2) $, we have
\[w_{\sigma,\al}(x)=\frac{\Psi_{\sigma,d}(\alpha)}{|x|^{2\sigma}}+O\left(\frac{1}{|x|^{2\sigma+2}}\right), \quad x \in \Z^d,\] with
\[
\Psi_{\sigma,d}(\alpha)=4^{\sigma}\frac{\Gamma(\frac{d}{2}- \alpha+\sigma ) \Gamma(\alpha)}{\Gamma(\frac{d}{2}-\alpha)\Gamma(\alpha- \sigma)}.
\]
Moreover, the Landau symbol admits a constant that is uniform over all $\sigma$ and $\alpha$ in the given range. \\
Furthermore, the function $\Psi_{\sigma, d}$ is strictly log-concave on $(\sigma, d/2)$,   strictly increasing on $(\sigma,\alpha_0)$ and strictly decreasing on $(\alpha_0, d/2)$ where $\alpha_0 = \frac{d/2 + \sigma}{2}$. \\
Asymptotically,
\begin{align*}
	\Psi_{\sigma,d}(\alpha) =\Psi_{\sigma,d}(\alpha_0)+O((\alpha-\alpha_0)^2)
\end{align*}
near $\alpha_0$.
\end{pro}
\begin{proof}
The asymptotic behavior of $w_{\sigma, \alpha}$ follows directly from Theorem~\ref{thm:RieszAsymptotics}. \\
Denote by $\psi$ the digamma function, i.e., the logarithmic derivative of the Gamma function, which is  strictly concave on $(0, \infty)$. 
We  compute  for $\Psi(\alpha)=\Psi_{\sigma,d}(\alpha)=4^{\sigma}\frac{\Gamma(2\alpha_0- \alpha ) \Gamma(\alpha)}{\Gamma(2\alpha_0-\sigma-\alpha)\Gamma(\alpha- \sigma)}$
\begin{align*}
\dfrac{d}{d\alpha}\ln\Psi (\alpha) &=  -\psi(2\alpha_0-\alpha)+\psi(\alpha)+\psi(2\alpha_0-\sigma-\alpha )-\psi(\alpha-\sigma) \\
& \quad =:F(\alpha-\alpha_0).
\end{align*}	
It follows immediately that 
 $(\Psi'/\Psi)(\alpha_0) = F(0)=0$ and, thus, $\Psi'(\alpha_0)=0$. Furthermore, setting $r=\alpha-\alpha_0$, we obtain
 \begin{align*}
F(r)
&=	  -\psi(\alpha_0-r)+\psi(\alpha_0+r)+\psi (\alpha_0-\sigma-r )-\psi(\alpha_0-\sigma+r),
\end{align*}	
for $r \in (\sigma-\alpha_0, d/2- \alpha_0)$. We observe 
\begin{align*}
	 \frac{d}{dr} F(r) &= \psi^{\prime}(\alpha_0 + r) - \psi^{\prime}(\alpha_0 + r - \sigma)  +\psi^{\prime}(\alpha_0 -r) - \psi^{\prime}(\alpha_0 - r - \sigma)<0,
\end{align*}
since $\psi$ is strictly concave on $(0, \infty)$ and, thus, $\psi'$ is strictly decreasing.
Therefore, $r=0$ is the only root of $F$ for $r \in (\sigma-\alpha_0, d/2- \alpha_0)$.
 We conclude that $\frac{d}{d\alpha}\ln\Psi (\alpha) $ is strictly positive for $\alpha \in (\sigma,\alpha_0)$ and strictly negative for $\alpha \in (\alpha_0, d/2)$. 
From that the claims about $\Psi$ follow, in particular, since
the signs of $\Psi^{\prime}(\alpha)$ and $\frac{d}{d\alpha} \ln \Psi(\alpha) = \Psi^{\prime}(\alpha)/\Psi(\alpha)$
coincide as $\Psi$ is positive. The asymptotic behavior near $\alpha_0$ is a Taylor approximation together with the  observation $\Psi'(\alpha_0)=0$ above.
\end{proof}

\section{Positive and Null-Criticality}\label{s:proof}
In this section we prove our main result, Theorem~\ref{t:main_result}. The proof is based on characterizing positive criticality of the weights $ w_{\sigma,\al} $ for $ \al<\alpha_0 $ and then using the ground states $\kappa_{-\alpha}$ to construct a null-sequence  to conclude optimality at $\alpha_0$.

We use a characterization of positive criticality from \cite{hake2025optimal} for which we briefly recall the definition of positive critical Hardy weights on general graphs. We refer to \cite{KPP18,hake2025optimal} for more details. 
A graph $b$ over a countable set $X$ is a symmetric function $b:X\times X\to [0,\infty)$ with zero diagonal satisfying $\sum_{y\in X}b(x,y)<\infty$ for all $x\in X$. The graph gives rise to the formal Laplacian $\L$ acting on functions $f:X\to \R$ such that $\sum_y b(x,y)|f(y)|<\infty$ for all $x\in X$, via 
\[\L f(x)=\sum_{y\in X}b(x,y)(f(x)-f(y))\] 
and the associated quadratic form $\mathcal{Q}$ with domain $\mathcal{D}$ given by
\[\mathcal{Q}(f)=\frac{1}{2}\sum_{x,y\in X}b(x,y)(f(x)-f(y))^2, \quad \mathcal{D}=\{f:X\to \R \mid \mathcal{Q}(f)<\infty\}.\]
We let $\mathcal{D}_0$ be the closure of $C_c(X)$ in $\mathcal{D}$ with respect to the norm $\|f\|_{\mathcal{Q}}=\sqrt{\mathcal{Q}(f)+|f(o)|^2}$ for some fixed $o\in X$. 
A function $u $ is called \emph{superharmonic} if $\L u\ge 0$ pointwise. A  function $w:X\to [0,\infty)$ is called a \emph{Hardy weight} if
\[\mathcal{Q}(\varphi)\ge \sum_{x\in X}w(x)\varphi(x)^2\]
for all $\varphi\in C_c(X)$. Whenever the inequality holds, it can be extended to all $\varphi\in \mathcal{D}_0$. Furthermore, we call the graph \emph{transient} if  there exists a non-trivial Hardy weight which is equivalent to existence of a  Green's function $G:X\times X\to (0,\infty)$ associated with $\L$ which gives rise to a posivity preserving  Green's operator $G$ associated with $\L$. Transience is equivalent to the presence of a non-trivial Hardy inequality, see e.g.\@ \cite{F00} or \cite[Theorem~6.1]{KLW}. A Hardy weight $w$ is called \emph{critical} if for every Hardy weight $w'\ge w$ we have $w'=w$. This is equivalent to the existence of a null sequence, i.e., a sequence $(e_n$) in $C_c(X)$ such that $ e_n $ converges pointwise to a positive function and $(Q-w)(e_n)\to0$ as $n\to\infty$, \cite[Theorem~5.3]{KPP20}.  A critical Hardy weight $ w $ is called \emph{positive critical} if  there exists a  function $ v\in \mathcal{D}_0 $, $v \neq 0$, such that $(\mathcal{Q}-w)(v)=0 $, i.e., the Hardy inequality allows for a minimizer in $\mathcal{D}_0$. Otherwise a critical Hardy weight is called \emph{null-critical}.

\begin{thm}[Theorems~2.15 and 2.16 of \cite{hake2025optimal}]\label{t:char_poscrit} The following statements are equivalent.
	\begin{itemize}
		\item[(i)]	$w$ is a positive critical Hardy weight.
		\item[(ii)] $w=\L v/v$ for some strictly positive superharmonic $v\in \mathcal{D}_0$.
		\item[(iii)] $w=\L Gk/Gk=k/Gk$ for some positive nontrivial $k  $ such that $\sum_{ X}k Gk <\infty$.
	\end{itemize}
\end{thm}

We return to the fractional Laplacian on $X=\Z^d$ together with the Hardy weights $w_{\sigma,\alpha}$ discussed in the previous section. 
The first step is to characterize the positive criticality of  $ w_{\sigma,\al} $.

\begin{thm}[Positive critical weights]\label{t:main_poscrit}
	Let $ d\in \N $,    $ \sigma\in (0,  1] $, $\sigma<d/2$ if $d \in \{1,2\}$, $\alpha\in(\sigma,d/2)$ and $\alpha_0= \frac{d/2 +\sigma}{2} $.  
	The weight $ w_{\sigma,\al}=\frac{\kappa_{\sigma-\al}}{\kappa_{-\al}} $ is a positive critical Hardy weight if and only if $\alpha< \alpha_0 $.
\end{thm}

We need the following lemma to prove Theorem~\ref{t:main_poscrit}.

\begin{lem}\label{l:G2condition}
Let $ d\in \N $,    $ \sigma\in (0,  1] $, $\sigma<d/2$ if $d \in \{1,2\}$, $\alpha\in(\sigma,d/2)$ and $\alpha_0= \frac{d/2 +\sigma}{2} $. Then, the function $ \kappa_{\sigma-\alpha} $ satisfies $\sum_{\Z^d}\kappa_{\sigma-\alpha}G^\sigma\kappa_{\sigma-\alpha}<\infty$    if and only if  $ \alpha< \alpha_0 $.
\end{lem}
\begin{proof}
By Proposition~\ref{p:exponents} and the asymptotic behavior of $ \kappa_{\beta} $, $ \beta\in \R $, given in Theorem~\ref{thm:RieszAsymptotics}, we have
\begin{align*}
	\sum_{\Z^d}\kappa_{\sigma-\alpha}G^\sigma\kappa_{\sigma-\alpha}
=\sum_{ \Z^d}\kappa_{\sigma-\alpha}\kappa_{-\alpha}
\asymp 
\sum_{x\in \Z^d\setminus\{0\}}\frac{1}{|x|^{2d-4\alpha+2\sigma}},
\end{align*}
where the symbol $\asymp$ stands for two sided estimates by positive constants. 
The latter sum converges if and only if $ 2d-4\alpha+2\sigma>d $, i.e., $ \alpha<(d/2+\sigma)/2=\alpha_0 $. This finishes the proof.
\end{proof}

\begin{proof}[Proof of Theorem~\ref{t:main_poscrit}]
By Proposition~\ref{p:exponents}  $ w_{\sigma,\al} 
=  \kappa_{\sigma-\alpha}/G^\sigma  \kappa_{\sigma-\alpha} $ is a Hardy weight for all $ \al>\sigma $.  Using Lemma~\ref{l:G2condition}, we find that $ \kappa_{\sigma-\alpha} $ satisfies the summability condition of Theorem~\ref{t:char_poscrit}~(iii) if and only if $\alpha < \alpha_0$. The implication (iii)$\Rightarrow$(i) of the just mentioned theorem finishes the proof. 
\end{proof}

From the theorem above and the monotonicity of the asymptotic constant $\Psi_{\sigma,d}(\alpha)$ we can now deduce our main result.

\begin{thm}[Null-critical weights]\label{t:main}
	Let $ d\in \N $,   $ \sigma\in (0,  1] $, $\sigma<d/2$ if $d \in \{1,2\}$ and $\alpha_0= \frac{d/2 +\sigma}{2} $. Then, the weight $ w_{\sigma,\al}=\frac{\kappa_{\sigma-\al}}{\kappa_{-\al}} $ is a null-critical Hardy weight if and only if  $\alpha=\alpha_0 $.
\end{thm}
\begin{proof}
	By Proposition~\ref{p:exponents}, we have  $\kappa_{-\alpha}=G^\sigma\kappa_{\sigma-\alpha}$ and by  Lemma~\ref{l:G2condition}, we have that $ \kappa_{\sigma-\alpha} $ satisfies $\sum_{\Z^d}\kappa_{\sigma-\alpha}G^\sigma\kappa_{\sigma-\alpha}<\infty$  if and only if  $ \alpha<\alpha_0 $.  By Theorem~\ref{t:char_poscrit} (iii)$\Rightarrow$(ii) we infer  $\kappa_{-\alpha}\in \mathcal{D}_0^\sigma$ for $\alpha<\alpha_0$. 
Thus, denoting the quadratic form $Q^\sigma$ of $\Delta^\sigma$, observing that the Hardy inequality extends to $ \mathcal{D}_0^\sigma$ and equality is realized by the ground states $\kappa_{-\alpha}$, we obtain
\begin{align*}
	0&\leq (Q^\sigma-w_{\sigma,\al_0})(\kappa_{-\alpha})=(Q^\sigma-w_{\sigma,\al})(\kappa_{-\alpha})+(w_{\sigma,\al}-w_{\sigma,\al_0})(\kappa_{-\alpha})\\
	&=\sum_{\Z^d}(w_{\sigma,\al}-w_{\sigma,\al_0})\kappa_{-\alpha}^2.
	\end{align*}
By Proposition~\ref{t:asymptotics_weight} and Theorem~\ref{thm:RieszAsymptotics} we obtain from the asymptotics of $w_{\sigma,\alpha}$ and $\kappa_{-\alpha}$
\begin{align*}
\ldots	=\left(\Psi_{\sigma,d}(\alpha)-\Psi_{\sigma,d}(\alpha_0)\right)\sum_{x\in \Z^d}\left[\frac{1}{|x|^{2\sigma+2d-4\alpha}}+O\left(\frac{1}{|x|^{2\sigma+2d-4\alpha+2}}\right)\right],
\end{align*}
where the constants in the Landau symbol are uniform in $\alpha$, cf. Theorem~\ref{thm:RieszAsymptotics}.
Estimating the sum by an integral and using polar coordinates yields that the sum for 
$\alpha=\alpha_0-\varepsilon$ is of order $O(1/\varepsilon)$. Furthermore,   Proposition~\ref{t:asymptotics_weight} shows that  $ \Psi_{\sigma,d}(\alpha_0-\varepsilon)-\Psi_{\sigma,d}(\alpha_0) =O(\varepsilon^2) $. So, in summary
\begin{align*}
	 (Q^\sigma-w_{\sigma,\al_0})(\kappa_{-(\alpha_0-\varepsilon)})=O(\varepsilon).
\end{align*}
Since  $\kappa_{-\alpha}$ is in $\mathcal{D}_0^\sigma$  and
$\kappa_{-\alpha}(x)\to \kappa_{-\alpha_0}(x)$ for all $ x\in \Z^d $  as $ \alpha\nearrow \alpha_0 $ by monotone convergence, 
we can extract a sequence $(e_n^\alpha)$ in $C_c(\Z^d)$ such that $ e_n^\alpha\to \kappa_{-\alpha} $, $\alpha<\alpha_0$ in the norm of $\mathcal{D}_0^\sigma$. By a diagonal sequence argument, we obtain a null-sequence which shows that $ w_{\sigma,\al_0} $ is a critical Hardy weight by \cite[Theorem~5.3]{KPP20}. By Theorem~\ref{t:main_poscrit}, $ w_{\sigma,\al_0} $ is not positive critical, so it must be null-critical.

To see that $w_{\sigma,\alpha}$ is not critical for $\alpha>\alpha_0$ we use the same argument as in \cite[Lemma 9]{KN}: if $w_{\sigma,\alpha}$ was critical for $\alpha>\alpha_0$, then it would have  to be null-critical by Theorem~\ref{t:main_poscrit}. Using the monotonicity of the constant $\Psi_{\sigma,d}(\alpha)<\Psi_{\sigma,d}(\alpha_0)$ and the fact that $w_{\sigma,\alpha_0}$ is null-critical, we infer that $w_{\sigma,\alpha}$ cannot be optimal at infinity which means that $w_{\sigma,\alpha}$ for $\alpha>\alpha_0$ is not null-critical, cf. e.g.~\cite[Lemma~2.12]{hake2025optimal}. 
\end{proof}

We summarize our findings in the following theorem.
\begin{thm}[Summary of the results]\label{t:summary}
Let $ d\in \N $,   and $ \sigma\in (0,  1] $, $\sigma<d/2$ if $d \in \{1,2\}$, and $\alpha_0= \frac{d/2 +\sigma}{2} $.  Then, the weight $ w_{\sigma,\al}=\frac{\kappa_{\sigma-\al}}{\kappa_{-\al}} $ is a Hardy weight for all $ \al\in (\sigma,d/2) $ and
	\begin{itemize}
		\item[(a)] for $\sigma <\alpha<\al_0 $, the weight $ w_{\sigma,\al} $ is positive critical,
		\item[(b)] for $ \al=\alpha_0 $, the weight $ w_{\sigma,\al} $ is null-critical,
		\item[(c)] for $ \alpha_0 <\al<d/2$, the weight $ w_{\sigma,\al} $ is subcritical.
	\end{itemize}
\end{thm}
\begin{proof}
This follows directly from Theorem~\ref{t:main_poscrit} and Theorem~\ref{t:main} and the fact that $ w_{\sigma,\al} $ is a Hardy weight for all $ \al\in (\sigma,d/2) $ by Proposition~\ref{p:exponents}.
\end{proof}

\begin{proof}[Proof of Theorem~\ref{t:main_result}]
	This follows from the theorem above and the asymptotic behaviour of $ w_{\sigma,\al} $ given in Proposition~\ref{t:asymptotics_weight}.
\end{proof}

\section{A Landis Type Theorem }\label{s:Landis}
As an application of our result we present a Landis type result. For the case $d=1$ such a result was shown in \cite[Theorem 5.2]{DKP}. A fractional Schrödinger operator  $\Delta^{\sigma}+V $ whose quadratic form on $\ell^2(\Z^d)$ is positive is called \emph{positive.}

\begin{thm} 
	Let $\sigma \in (0,1\wedge d/2)$.  
	 Assume that $u$ is a  harmonic function of a positive fractional Schr\"odinger operator $ \Delta^{\sigma}+V $ on $\mathbb{Z}^d$ with $ V\leq 1 $. If  
	$$  u\in O \left({|x|^{\sigma -d/2}}\right) \quad\mbox{ and }\quad \liminf_{x\to\infty} |u(x)| |x|^{d+2\sigma}=0, $$
	then $ u=0 $.
\end{thm}
\begin{proof}
	By Theorem~\ref{t:main} there exists an optimal Hardy weight $w_{\sigma,\alpha_0}$ such that the associated Agmon ground state $\kappa_{-\alpha_0}$ of $\Delta^{\sigma} - w_{\sigma,\alpha_0}$ on $\mathbb{Z}^d$ satisfies $\kappa_{-\alpha_0} \asymp|x|^{2\alpha_0 -d} =|x|^{\sigma -d/2}$ according to Theorem~\ref{thm:RieszAsymptotics} and since $\alpha_0=(d/2+\sigma)/2$ . Thus, due to our assumption, $u \in O(\kappa_{-\alpha_0})$.
	
	On the other hand, we have that the Green's function $G_1^{\sigma}=(\Delta^{\sigma}+1)^{-1}1_0$ of the operator $\Delta^{\sigma}+1$ on $\mathbb{Z}^d$ behaves as $G_1^{\sigma} \asymp |x|^{-(d+2\sigma)}$ \cite[Theorem~7]{DMRM}. Hence, \cite[Corollary~3.3]{DKP} implies that  $u=0$ on $\mathbb{Z}^d$.
\end{proof}

\appendix
\section{Asymptotics of the Riesz Kernel on $\Z^d$}

We study the kernel of the discrete fractional Laplacian $\Delta^\sigma$ based on the standard Laplacian $\Delta = \Delta^1$ on $\Z^d$.
This kernel $\kappa_\sigma$ for $-d/2 < \sigma < 1$, is called the  Riesz kernel and is defined as $\kappa_\sigma(0) = 0$ for $\sigma > 0$,  $\kappa_0=1_0$, and otherwise by the integral representation
\[
\kappa_\sigma(x) = \frac{1}{|\Gamma(-\sigma)|} \int_0^\infty e^{-t\Delta}1_0(x) \, \frac{dt}{t^{1+\sigma}}.
\]
Indeed, there are various sources where the asymptotic behavior of $\kappa_\sigma$ is studied, e.g.,   \cite{CR18,CRSTV,DMRM,FRR,GRM,Slade}.  So, for the sake of completeness and since to the best of our knowledge the asymptotics with the precise order of the error  as given below do not appear in the literature so far, we provide a proof  in the following theorem.

\begin{thm}[Asymptotics of the Riesz kernel]\label{thm:RieszAsymptotics}
	For $-d/2 < \sigma < 1$, $\sigma \neq 0$, the Riesz kernel is finite  and satisfies for $|x| \to \infty$ 
	\[
	\kappa_\sigma(x) = C_{d,\sigma} \, |x|^{-d-2\sigma} + O\left(|x|^{-d-2\sigma-2}\right),
	\]
	where
	\[
	C_{d,\sigma} =
	\frac{4^{\sigma} \Gamma(d/2+\sigma)}{\pi^{d/2}|\Gamma(-\sigma)|}.
	\]
	The constant in the Landau symbol can be chosen uniformly over the range of $\sigma$.
\end{thm}

\begin{rem}
	Observe that the condition $\sigma > -d/2$ is necessary for the integral in the definition of  $\kappa_{-\sigma}$ to converge at infinity since the heat kernel decays as $t^{-d/2}$ for large $t$. The condition $\sigma < 1$ is necessary for convergence at $0$ since the heat kernel is bounded below by a constant times $t$ for small $t$. This will be made precise in the proof.
\end{rem}

In the following $ C>0$   always denote  constants which may change from line to line.

\begin{lem}[Heat kernel bounds]
	Let $d \in \mathbb{N}$. The heat kernel on $\Z^d$ satisfies the following bounds:
	\begin{enumerate}
		\item[(a)] For all $x \neq 0$, there is $C>0$ such that for all $0<t<|x|_\infty$, we have $e^{-t\Delta}1_0(x) \leq C  t^{|x|_{\infty}}$. Furthermore, 	
		there are $m_0 \in \N$ and $C>0$ such that for $x \in \Z^d$ with $|x| \geq m_0$ and $0 < t \leq  |x|^{1/(d+1)}$, we have
		\[
		e^{-t\Delta}1_0(x) \leq C  t^{|x|_{\infty}}  e^{-(1/d)|x|\log|x|}.
		\]
		\item[(b)] Fix $k \in \N$.
		For $t \to \infty$ and $x \in \Z^d $, one has
		\begin{multline*}
			e^{-t\Delta}1_0(x) = (4\pi t)^{-d/2} e^{-| x|^2/(4t)}\\
			+ O\left(t^{-d/2-1}\left[ \left(1+ \Big(\tfrac{|x|}{\sqrt{t}}\Big)^k\right) e^{-| x|^2/(4t)}    + t^{-(k-3)/2}    \right]\right).
		\end{multline*}
	\end{enumerate}
\end{lem}

\begin{proof}
	For (a) we use \cite[Theorem~3.5]{Pang} for the heat kernel $ K^{(1)}$ on $\Z$
	\begin{align*}
		K^{(1)}(t, 0, m) \leq C |m|^{-\frac{1}{2}} \exp\left( |m|F\left( \frac{|m|}{2t} \right) \right)
	\end{align*}
	for some constant $C>0$ and $m \in \Z$, where
	\begin{align*}
		F(y)  &= -\log\left(y + \sqrt{y^2 + 1}\right) + y^{-1}\left(\sqrt{y^2 + 1} - 1\right)\\
		 &= -\log(2y) + 1 + O\left(\frac{1}{y}\right).
	\end{align*}
	So, we apply this for  $y={{|m|}}/{2t} $. Assuming that $t/|m|$ is bounded by a constant which will be $\sqrt{d}$, we can choose $m_0$ large enough that for $|m|\ge m_0$
	\begin{align*}
		K^{(1)}(t, 0, m) &\leq e^{ {|m|}F\left( |{m}|/{2t} \right) } 
		  \le  C   t^{|m|}e^{-|m|(\log (|m|/e)+O(t/|m|))} \\
		&\le  C   t^{|m| }e^{-(1/\sqrt{d})|m|\log (\sqrt{d}|m|)} .
	\end{align*}
	 For higher dimension, we observe that the heat kernel on $\Z^d$ factorizes as $$e^{-t\Delta}1_0(x)=K^{(d)}(t,0,x) = \prod_{i=1}^d K^{(1)}(t,0,x_i).$$ 
	 We pick $|x_j|=|x|_\infty \ge (1/\sqrt{d})|x|$ for some $j=1,\ldots,d$. To apply the consideration for the one dimensional case above, we observe that by assumption  $t/|x_j|=t/|x_\infty|
	 \leq \sqrt{d} |x|^{-d/(d+1)} \leq \sqrt {d}$ for $x\neq 0$.	 
Since the heat kernel is bounded by one, we can bound all other factors by $1$. Thus, for $|x|\ge m_0$ and $t\leq  |x|^{1/(d+1)}$
\begin{align*}
		K^{(d)}(t,0,x) &\leq  K^{(1)}(t,0 ,x_j) \le C   t^{|x_j|}e^{-(1/\sqrt{d})|x_j|\log (\sqrt{d}|x_j|)}\\
		 &\le C   t^{|x|_\infty}e^{-(1/{d})|x|\log |x|},
	\end{align*}
	where we used again $|x|_\infty \geq |x|/\sqrt{d}$. The first statement of (a) follows even more easily along the same lines.

	For part (b), we apply \cite[Theorem~2.1.3~(2.9)]{LawlerLimic} which gives the asymptotic expansion of the heat kernel on $\Z^d$ for large $t$ and the simple random walk. That is we consider the normalized Laplacian $(1/2d)\Delta$ with the covariance matrix $\Gamma = (1/d) I_d$ which satisfies $\mathcal J^*(x)=x\cdot\Gamma^{-1}x=d^{1/2}|x|$, cf.~\cite[page~11]{LawlerLimic}. Thus, \cite[Theorem~2.1.3~(2.9)]{LawlerLimic} states that, for all $k \in \mathbb{N}$, we have
	\begin{multline*}    e^{-(t/2d)\Delta}1_0(x) = (2\pi t/d)^{-d/2} e^{-d| x|^2/(2t)}\\  + O\left(t^{-d/2-1}\left[ \left(1+\Big(\tfrac{|x|}{\sqrt{t}}\Big)^k\right) e^{-d| x|^2/(2t)}+ t^{-(k-3)/2}    \right]\right),    \end{multline*}
	for $t \to \infty$. Rescaling by replacing $t$ by $2dt$ gives the result.
\end{proof}

\begin{proof}[Proof of Theorem~\ref{thm:RieszAsymptotics}]
	We split the integral defining $\kappa_\sigma$ into two regions
	\begin{multline*}
		\kappa_\sigma(x)=\frac{1}{|\Gamma(-\sigma)|}\int_0^\infty e^{-t\Delta}1_0(x) \, \frac{dt}{t^{1+\sigma}} = \frac{1}{|\Gamma(-\sigma)|} \left(\int_{0}^{|x|^{\frac{1}{d+1}}}\ldots+\int_{|x|^{\frac{1}{d+1}}}^{\infty} \ldots\right)
	\end{multline*}
	and denote the two integrals by $I_1(x)$ and $I_2(x)$, respectively.

	Region I consists of $0 < t \leq |x|^{1/(d+1)} $. 
	Finiteness in this region is only a potential issue for $\sigma > 0$  as the heat kernel is bounded by 1. Furthermore, we have $\kappa_\sigma(0) = 0$ for $\sigma > 0$ by definition. For $x\neq0$, the first statement of (a) of the previous lemma gives an upper bound of $t$ on the heat kernel, so the integral converges at $0$.
	
For the asymptotics as $|x| \to \infty$, there is no loss of generality in assuming $|x|_{\infty} \geq \max\{2, m_0 \}$, where $m_0$ is chosen according to part~(a) of the previous lemma. The latter gives that the heat kernel is bounded by
	\begin{align*}
		e^{-t\Delta}1_0(x) \leq C t^{|x|_{\infty}}  e^{(-1/d)|x| \log|x|}.
	\end{align*}
	This yields with $|x|_{\infty} \leq |x|$
	\begin{align*}
		I_1(x) &= \int_0^{|x|^{\frac{1}{d+1}}} e^{-t\Delta}1_0(x) \frac{dt}{t^{1+\sigma}} \\
		&\leq C e^{-(1/d)|x| \log |x|}\left[ \int_0^{1} t^2
		\frac{dt}{t^{1+\sigma}}+ \int_1^{|x|^{\frac{1}{d+1}}} t^{|x|} \frac{dt}{t^{1+\sigma}}\right]  \\\
		&= \,\, O \left( e^{-(1/d)|x| \log |x| }\, e^{(1/(d+1))(|x|-\sigma
		) \log |x|}  \right),
	\end{align*}
	which  is superpolynomially small for $|x|\to\infty$. Moreover, the constant in front of the integral over $0 \leq t \leq 1$ is bounded by $C(2-\sigma)^{-1}$, and since  $\sigma < 1$, 
	 one can find a constant for the Landau symbol that is independent of $\sigma$.

	Region II consists of $t >  |x|^{1/(d+1)}$. Here we use   the asymptotic given by (b) of the lemma above given for $k \in \N$ as 
	\begin{align*}
		e^{-t\Delta}1_0(x) &= \frac{ e^{-| x|^2/(4t)}}{(4\pi t)^{d/2}}
		+ O\left(t^{-d/2-1}\left[ \left(1+\Big(\tfrac{|x|}{\sqrt{t}}\Big)^k\right) e^{-| x|^2/(4t)}       \right]\right) \\
		& \quad + O\left(t^{-(d+k-1)/2}\right).
	\end{align*}
	Integrating term by term, the integral  $I_2(x) = \int_{t \geq |x|^{1/(d+1)}}  e^{t\Delta}1_0(x)dt/t^{1+\sigma}$ splits into three summands which we denote by $J_1(x)$, $J_2(x)$, and $J_3(x)$.
	
	We start by computing the main term $J_1(x)$ which gives with the change of variables  $u = \frac{|x|^2}{4t}$, so $t = \frac{|x|^2}{4u}$, $dt = -\frac{|x|^2}{4u^2}du$
	\[
	J_1(x) = \int_{|x|^{\frac{1}{d+1}}}^\infty \frac{e^{-|x|^2/(4t)}}{(4\pi t)^{d/2}} \frac{dt}{t^{1+\sigma}}
	=\frac{1}{(4\pi)^{d/2}}  \frac{4^{d/2+\sigma} }{|x|^{d+2\sigma}}\int_{0}^{4^{-1}|x|^{\frac{2d+1}{d+1}}} e^{-u} u^{d/2+\sigma-1}du.
	\]
	Using the definition of the Gamma function, 
	\(
	\Gamma(d/2+\sigma)=\int_0^\infty e^{-u} u^{d/2+\sigma-1} du 
	\),
	we obtain
	\begin{align*}J_1(x) &= \frac{4^{\sigma} \Gamma(d/2+\sigma)}{\pi^{d/2}} |x|^{-d-2\sigma} 
		-\frac{1}{\pi^{d/2}}  \frac{4^{\sigma} }{|x|^{d+2\sigma}}\int_{4^{-1}|x|^{\frac{2d+1}{d+1}}}^\infty e^{-u} u^{d/2+\sigma-1} du
		\\&=\frac{4^{\sigma} \Gamma(d/2+\sigma)}{\pi^{d/2}} |x|^{-d-2\sigma} +O(e^{-c|x|})
	\end{align*}
	for some $c>0$ that can be chosen independently of $\sigma$.
	
	The first error term is handled in a similar manner. 
	We split the integral  $J_2(x)$ into two integrals $J_{2,1}(x)$ and $J_{2,2}(x)$ according to the error terms of order $t^{-d/2-1}$ and $t^{-d/2-1}  (|x|/\sqrt t)^{k}$, respectively.
	The first part $J_{2,1}(x)$ of $J_2(x)$ is similar to the main term and gives $O(|x|^{-d-2\sigma-2})$   since we have an extra $t^{-1}$ which substitutes into an extra $|x|^{-2}$. The second part is
	\begin{align*}
		J_{2,2}(x) &= C\int_{|x|^{\frac{1}{d+1}}}^\infty t^{-d/2-1} \left(\tfrac{|x|}{\sqrt{t}}\right)^{k} e^{-|x|^2/(4t)} \frac{dt}{t^{1+\sigma}}\\
		&= C |x|^{k} \int_{|x|^{\frac{1}{d+1}}}^\infty t^{-d/2-k/2-2-\sigma} e^{-|x|^2/(4t)} dt\\
		&= C \frac{4^{d/2+k/2+1+\sigma}}{|x|^{d+2+2\sigma}} \int_{0}^{4^{-1}|x|^{\frac{2d+1}{d+1}}} e^{-u} u^{d/2 + k/2 + \sigma} du \\
		&= O(|x|^{-d-2\sigma-2}),
	\end{align*}
	where in the second to last step we changed variables  $u = |x|^2/(4t)$, as before. 
	Again, note that the constant for the Landau symbol can be chosen independently of $\sigma$.
	
	The final error $J_3(x)$, which is the integral of $O\left(t^{-(d+k-1)/2-1-\sigma}\right)$  over $t\ge |x|^{1/(d+1)}$, gives $O(|x|^{-d-4})$ for a suitable choice of $k$ such that the constant  in the Landau symbol can be chosen independently of $\sigma$. 
		
	Thus, combining everything we have, 
	\begin{align*}
		\kappa_\sigma(x) 
		&=\frac{1}{|\Gamma(-\sigma)|} \left(I_1(x) + J_1(x) + J_{2,1}(x) + J_{2,2}(x)+ J_3(x)\right)\\
		&= C_{d,\sigma} |x|^{-d-2\sigma} + O\left( |x|^{-d-2\sigma-2} \right),
	\end{align*}
	with $C_{d,\sigma}={4^{\sigma} \Gamma(d/2+\sigma)}/({\pi^{d/2}|\Gamma(-\sigma)|}) $ as stated. Since all constants in the Landau symbols of the error terms can be chosen independently of $\sigma$ and since 
	$|\Gamma(-\sigma)|$ is bounded from below for $\sigma < 1$, we find that the latter Landau symbol also admits a constant that is uniform over $\sigma$ in the given range. 
	This completes the proof.
\end{proof}

\textbf{Acknowledgements.} The authors acknowledge the financial support of the DFG. MK and FP acknowledge the financial support and hospitality of the IIAS, Jerusalem. We are grateful to Ujjal Das, Bartek Dyda, Florian Fischer, Shubham Gupta, David Krej\v{c}i\v{r}\'{i}k, Marius Nietschmann, Yehuda Pinchover  and  Luz Roncal
for their comments on a first version of this paper. 

\bibliographystyle{alpha}
\bibliography{mynewbib}

\end{document}